\documentclass[12pt]{article}

\usepackage{amssymb,amsmath,amsthm}

\newcommand{\ovl}{\overline}
\newcommand{\n}{\noindent}

\newcommand{\cl}[1]{{\mathcal{#1}}}
\newcommand{\bb}[1]{{\mathbb{#1}}}

\numberwithin{equation}{section}

\theoremstyle{plain}
\newtheorem{lem}{Lemma}[section]
\newtheorem{pro}[lem]{Proposition}
\newtheorem{thm}[lem]{Theorem}

\theoremstyle{definition}

\theoremstyle{remark}
\newtheorem{rem}[lem]{Remark}

\begin{document}

\null\vspace{.5in}
\setcounter{page}{1}
\begin{LARGE}
\begin{center}
VANISHING OF SECOND COHOMOLOGY\\ FOR TENSOR PRODUCTS OF\\ TYPE II$_1$ VON NEUMANN ALGEBRAS
\end{center}
\end{LARGE}
\bigskip

\begin{center}
\begin{tabular}{cc}
{\large Florin Pop}&{ \large Roger R. Smith}\\
&\\
Department of Mathematics & Department of Mathematics\\
and Computer Science& Texas A\&M University\\
Wagner College & College Station, TX 77843\\
Staten Island, NY 10301 &{} \\
&\\
fpop@wagner.edu & rsmith@math.tamu.edu
\end{tabular}
\end{center}

\vspace{0.5in}

\begin{abstract}
 We show that the second cohomology group $H^2(M\ovl\otimes N, M\ovl\otimes N)$ is always zero for arbitrary type II$_1$ von Neumann algebras $M$ and $N$.
\end{abstract}

\vspace{0.5in}


\vspace{0.5in}

\newpage

\section{Introduction}\label{sec1}

\indent 

The theory of bounded Hochschild cohomology for von Neumann algebras was initiated by Johnson, Kadison and Ringrose in a series of papers \cite{JKR,KR1,KR2}, which laid the foundation for subsequent developments. These were a natural outgrowth of the theorem of Kadison \cite{K} and Sakai \cite{S} which established that every derivation $\delta\colon \ M\to M$ on a von Neumann algebra $M$ is inner; $H^1(M,M)=0$ in cohomological terminology. While cohomology groups can be defined for general $M$-bimodules (see Section \ref{sec2} for definitions), this derivation result ensured special significance for $M$ as a bimodule over itself. When $M$ is represented on a Hilbert space $H$, then $B(H)$ is also an important $M$-bimodule, but here the known results are less definitive. For example, it is not known whether every derivation $\delta\colon \ M\to B(H)$ is inner, a problem known to be equivalent to the similarity problem \cite{Ki}. In \cite{KR2}, it was shown the $H^n(M,M)=0$ for all $n\ge 1$ when $M$ is an injective von Neumann algebra, a class which includes the type I algebras. They conjectured that this should be true for all von Neumann algebras, now known as the Kadison--Ringrose conjecture. The purpose of this paper is to verify this for the second cohomology of tensor products of type II$_1$ von Neumann algebras.

The study of this conjecture reduces to four cases in parallel with the type decomposition of Murray and von Neumann. Three of these are solved. As noted above, the type I case was determined at the outset of the theory, while the types II$_\infty$ and III cases were solved by Christensen and Sinclair \cite{CS3} after they had developed the theory of completely bounded multilinear maps \cite{CS1} and applied it, jointly with Effros \cite{CES}, to cohomology into $B(H)$. They showed that $H^n_{cb}(M,M)=0$, $n\ge 1$, for all von Neumann algebras, where the subscript indicates that all relevant multilinear maps are required to be completely bounded. Since then, all progress has hinged on reducing a given cocycle to one which is completely bounded and then quoting their result. In this paper we follow a different path, although complete boundedness will play an important role.

The one remaining open case is that of type II$_1$ von Neumann algebras. There are several positive results for special classes:\ the McDuff factors \cite{CS3}, those factors with Cartan subalgebras \cite{PS,CPSS,SSam,Cam}, and those with property $\Gamma$ \cite{CS3,C3,CPSSgam}. While tensor products form a large class of type II$_1$ von Neumann algebras, the prime factors fall outside our scope. The best known examples are the free group factors, shown to be prime by Ge \cite{G}, and these do not lie in any of the classes already mentioned, so nothing is known about their cohomology.

Section \ref{sec2} gives a brief review of definitions and some results that we will need subsequently. The heart of the paper is Section \ref{sec3} where we prove that $H^2(M\ovl\otimes N, M\ovl\otimes N) = 0$ for separable type II$_1$ von Neumann algebras. This restriction is made in order to be able to choose certain special hyperfinite subalgebras that are only available in this setting. The proof proceeds through a sequence of lemmas which reduce a given cocycle to one with extra features, after which we can exhibit it as a coboundary. In this process, particular use is made of complete boundedness and the basic construction, \cite{J}, for containments of type II$_1$ algebras. Section \ref{sec4} handles the general case by deducing it from the separable situation using several known techniques to be found in \cite[\S 6.5]{SSbook}. However, Lemma \ref{lem4.1} appears to be new.

For general background on cohomology we refer to the survey article \cite{Rs} and the monograph \cite{SSbook}. The theory of complete boundedness is covered in several books \cite{ER,Pa,Pis}, while \cite{SSfbook} contains an introduction to the basic construction algebra.

\section{Preliminaries and notation}\label{sec2}

\indent 

Since this paper is only concerned with second cohomology, we will only give the definitions at this level, referring to \cite{SSbook} for the general case. Let $A$ be a $C^*$-algebra with an $A$-bimodule $V$. For $n=1,2,3$, ${\cl L}^n(A,V)$ denotes the space of bounded $n$-linear maps from $A \times\cdots\times A$ into $V$, while ${\cl L}^0(A,V)$ is defined to be $V$. For $v\in V$ and $\phi_n\in {\cl L}^n(A,V)$, $n=1,2$, the coboundary map $\partial\colon \ {\cl L}^n(A,V)\to {\cl L}^{n+1}(A,V)$ is defined as follows:
\begin{align}\label{eq2.1}
 \partial v(a) &= va-av,\qquad a\in A,\\
\label{eq2.2}
\partial\phi_1(a_1,a_2) &= a_1\phi_1(a_2) - \phi_1(a_1a_2) + \phi_1(a_1)a_2,\qquad a_i\in A,\\
\label{eq2.3}
\partial\phi_2(a_1,a_2,a_3) &= a_1\phi_2(a_2,a_3) -  \phi_2(a_1a_2,a_3) + \phi_2(a_1,a_2a_3) + \phi_2(a_1,a_2)a_3,\qquad a_i\in A,
\end{align}
An algebraic calculation gives $\partial\partial = 0$. Cocycles are those maps $\phi$ for which $\partial\phi= 0$, while coboundaries are maps of the form $\partial\xi$. The $n^{\text{th}}$ cohomology group $H^n(A,V)$ is then the quotient space of $n$-cocycles modulo $n$-coboundaries. In particular, $H^1(A,V)$ is the space of derivations modulo inner derivations. Since we plan to prove that certain second cohomology groups are zero, this amounts to showing that each 2-cocycle is a 2-coboundary. There is a considerable theory of cohomology, much of which is summarized in \cite{SSbook}. We use this monograph as our standard reference, but include some results below which are not to be found there. The first two of these concern complete boundedness, the second of which is a small extension of the factor case  of \cite{CS3} (the results of this paper appear in \cite{SSbook}).

\begin{lem}\label{lem2.1}
Let $M\subseteq B(H)$ and $S\subseteq B(K)$ be von Neumann algebras with $S$ hyperfinite. If $\phi\colon \ M\ovl\otimes S\to B(H\otimes_2 K)$ is bounded, normal and $(I\otimes S)$-modular, then $\phi$ is completely bounded.
\end{lem}

\begin{proof}
We regard $M$ and $S$ as both represented on $B(H\otimes_2 K)$. The $(I\otimes S)$-modularity implies that the restriction $\psi$ of $\phi$ to $M\otimes I$ maps into $S'$, so $\phi$ maps the minimal tensor product $M \otimes_{\text{min}} S$ into $C^*(S',S)$. Now $S$ contains arbitrarily large matrix subfactors ${\bb M}_n$, $n\ge~1$, and $\phi|_{M\otimes {\bb M}_n}$ can be regarded as the composition of $\psi\otimes \text{id}_n\colon \ M\otimes {\bb M}_n\to S'\otimes {\bb M}_n$ with a $*$-isomorphism $\pi_n\colon \ S'\otimes {\bb M}_n \to C^*(S', {\bb M}_n)$. The uniform bound $\|\phi\|$ on each of these restrictions then shows that $\psi$ is completely bounded. Hyperfiniteness of $S$ gives a $*$-homomorphism $\rho\colon\ S'\otimes_{\min} S\to C^*(S',S)$ defined on elementary tensors by $s'\otimes s\mapsto s's$, \cite[Proposition 4.5]{EL}, so $\phi|_{M\otimes_{\min} S}$ is the composition $\rho \circ (\psi\otimes \text{id}_S)$, showing complete boundedness on $M\otimes_{\min} S$. The same conclusion on $M \ovl\otimes S$ now follows from normality of $\phi$ and the Kaplansky density theorem applied to $M \otimes_{\min} S \otimes {\bb M}_n \subseteq M\ovl\otimes S\otimes {\bb M}_n$.
\end{proof}

The proof that we have given of this result relies on normality and hyperfiniteness, and it would be interesting to know if it holds without these restrictions. The next result is known for factors \cite{CS3}, but does not appear to be in the literature in the generality that we require.

\begin{lem}\label{lem2.2}
Let $M$ and $S$ be type {\rm II}$_1$ von Neumann algebras with $S$ hyperfinite, and let $R\subseteq M$ be a hyperfinite von Neumann subalgebra with $R'\cap M  = Z(M)$. Let $\phi\colon \ (M\ovl\otimes S) \times (M\ovl \otimes S) \to M\ovl\otimes S$ be a bounded separately normal bilinear map which is $R\ovl\otimes S$-multimodular. Then $\phi$ is completely bounded.
\end{lem}

\begin{proof}
Lemma 5.4.5 (ii) of \cite{SSbook} yields the inequality
\begin{equation}\label{eq2.4}
 \left\|\sum^n_{i=1} \phi(x_i,y_i)\right\|\le 2\|\phi\| \left\|\sum^n_{i=1} x_ix^*_i\right\|^{1/2} \left\|\sum^n_{i=1} y^*_iy_i\right\|^{1/2}
\end{equation}
for arbitrary finite sets of elements $x_i,y_i\in M\ovl\otimes S$, $1\le i\le n$. If $\phi_n$ denotes the $n^{\text{th}}$ amplification of $\phi$ to $M\ovl\otimes S\otimes {\bb M}_n$, then \eqref{eq2.4} says that
\begin{equation}\label{eq2.5}
 \|\phi_n(R,C)\| \le 2\|\phi\| \|R\|\|C\|
\end{equation}
for operators $R$ and $C$ in the respective row and column spaces $\text{Row}_n(M\ovl\otimes S)$ and $\text{Col}_n(M\ovl\otimes S)$. Now $S$ contains arbitrarily large matrix subfactors and so has no finite dimensional representations. From \cite[Proposition 3.4]{PSS}, $S$ norms $M\ovl\otimes S$. Thus, for each pair $X,Y \in {\bb M}_n(M\ovl\otimes S)$,
\begin{align}
\|\phi_n(X,Y)\| &= \sup\{\|R\phi_n(X,Y)C\|\colon \ R\in \text{Row}_n(I\ovl\otimes S), C\in \text{Col}_n(I \ovl\otimes S), \|R\|, \|C\|\le 1\}\notag\\
&= \sup\{\|\phi_n(RX,YC)\|\colon \ R\in \text{Row}_n(I\ovl\otimes S), C\in \text{Col}_n(I\ovl\otimes S), \|R\|, \|C\|\le 1\}\notag\\
\label{eq2.6}
&\le 2\|\phi\| \|X\| \|Y\|,
\end{align}
where the first equality uses $(I\ovl\otimes S)$-modularity and the final inequality is \eqref{eq2.5} applied to the row $RX$ and the column $YC$. Since $n$ was arbitrary in \eqref{eq2.6}, complete boundedness of $\phi$ is established by this inequality.
\end{proof}

In \cite{Rd}, it was shown that, for von Neumann algebras $M\subseteq B(H)$, every derivation $\delta\colon \ M\to B (H)$ is automatically bounded and ultraweakly continuous. We will require two further facts about derivations which we quote from the work of Christensen in the next two lemmas. In the first one, our statement is extracted from the proof of $[4\Rightarrow 2]$ in the referenced theorem. 

\begin{lem}[Theorem 3.1 in \cite{C2}]\label{lem2.3}
Each completely bounded derivation $\delta\colon \ M\to B(H)$ is inner and is implemented by an operator in $B(H)$.
\end{lem}

\begin{lem}[special case of Theorem 5.1 in \cite{C1}]\label{lem2.4}
If $M\subseteq N$ is an inclusion of finite von Neumann algebras, then each derivation $\delta\colon \ M\to N$ is inner and is implemented by an element of $N$.
\end{lem}

\section{Separable algebras}\label{sec3}

\indent 

In this section we will prove the vanishing of second cohomology for tensor products of type II$_1$ von Neumann algebras under the additional hypothesis that each algebra is separable. In this context, separability of a von Neumann algebra means the existence of a countable ultraweakly dense subset or, equivalently, a faithful normal representation on a separable Hilbert space. If $M$ is a separable type II$_1$ factor, then it was shown in \cite{Po} that $M$ has a maximal abelian subalgebra (masa) $A$ and a hyperfinite subfactor $R$ such that $A \subseteq R\subseteq M$ and $R'\cap M = {\bb C}1$. This was generalized to separable type II$_1$ von Neumann algebras in \cite{SSsc} with the modifications that $R$ is now a hyperfinite von Neumann subalgebra and that $R'\cap M$ is now the center $Z(M)$. Separability is essential for these results and this is the reason for restricting to separable algebras in this section. Throughout we assume that $M$ and $N$ are separable type II$_1$ von Neumann algebras with respective centers $Z(M)$ and $Z(N)$. We fix choices of masas $A$ and $B$ and hyperfinite subalgebras $R$ and $S$ so that
\begin{equation}\label{eq3.1}
 A \subseteq R \subseteq M,\quad B \subseteq S \subseteq N,
\end{equation}
and
\begin{equation}\label{eq3.2}
 R' \cap M = Z(M),\quad S'\cap N  = Z(N).
\end{equation}
We also note the trivial fact that centers are always contained in masas.

We wish to consider a bounded 2-cocycle $\phi\colon \ (M\ovl\otimes N) \times (M\ovl\otimes N)\to M\ovl\otimes N$ and show that it is a coboundary. The general reduction results of \cite[Ch.~3]{SSbook} allow us to impose the following extra conditions on $\phi$:
\begin{itemize}
 \item[(C1)] $\phi$ is separately normal in each variable;
\item[(C2)] $\phi(x,y)=0$ whenever $x$ or $y$ lies in $R\ovl\otimes S$;
\item[(C3)] $\phi$ is $R\ovl\otimes S$-multimodular. 
\end{itemize}
The latter condition is a consequence of (C2), from \cite[Lemma 3.2.1]{SSbook}, so (C2) is a slightly stronger requirement.
We begin by making a further reduction. 

\begin{lem}\label{lem3.1}
 Let $\phi$ be a 2-cocycle on $M\ovl\otimes N$ which satisfies conditions (C1)--(C3). Then  $\phi$ is equivalent to a 2-cocycle $\psi$ on $M\ovl\otimes N$ satisfying (C1)--(C3) and the additional condition
\begin{equation}\label{eq3.3}
 \psi(m_1\otimes I, m_2\otimes I) = \psi(I\otimes n_1, I\otimes n_2) = 0
\end{equation}
for $m_1,m_2\in M,\  n_1,n_2\in N$.
\end{lem}

\begin{proof}
Multimodularity with respect to $I\otimes S$ shows that
\begin{align}
 (I\otimes s) \phi(m_1\otimes I, m_2\otimes I) &= \phi(m_1\otimes s, m_2\otimes I)
= \phi(m_1 \otimes I, m_2 \otimes s)\notag\\
\label{eq3.4}
&= \phi(m_1\otimes I, m_2\otimes I)(I\otimes s),\qquad m_1,m_2\in M, s\in S,
\end{align}
from which it follows that $\phi(m_1\otimes I, m_2\otimes I) \in (I\otimes S)'\cap (M\ovl\otimes N) = M\ovl\otimes Z(N)$ for all $m_1,m_2\in M$. 
Note that $Z(N)\subseteq B\subseteq S$, and so $M\ovl{\otimes}Z(N)\subseteq M\ovl{\otimes}S$. Since
\begin{equation}\label{eq3.5}
 \phi(m_1\otimes s_1, m_2\otimes s_2) = \phi(m_1\otimes I, m_2\otimes I) (I\otimes s_1s_2)
\end{equation}
for $m_1,m_2\in M, s_1,s_2\in S$, we conclude that $\phi$ maps $(M\ovl\otimes S) \times (M\ovl\otimes S)$ to $M\ovl\otimes S$. Thus the restriction of $\phi$ to $M\ovl\otimes S$ is completely bounded by Lemma \ref{lem2.2}. It follows from \cite{CS3} that there is a normal $(R\ovl\otimes S)$-modular map $\alpha\colon \ M\ovl\otimes S\to M\ovl\otimes S$ such that $\phi|_{M\ovl\otimes S} = \partial\alpha$, and a similar argument gives a normal $(R\ovl\otimes S)$-modular map $\beta\colon \ R\ovl\otimes N \to R\ovl\otimes N$ such that $\phi|_{R\ovl\otimes N} = \partial\beta$. Using the normal conditional expectations  ${\bb E}_{M\ovl\otimes S}$ and ${\bb E}_{R\ovl\otimes N}$ of $M\ovl\otimes N$ onto $M\ovl\otimes S$ and $R\ovl\otimes N$ respectively, we now extend $\alpha$ and $\beta$ to $(R\ovl\otimes S)$-modular maps $\tilde\alpha,\tilde\beta \colon \ M\ovl\otimes N\to M\ovl\otimes N$ by
\begin{equation}\label{eq3.6}
 \tilde\alpha = \alpha \circ {\bb E}_{M\ovl\otimes S},\quad \tilde\beta = \beta\circ {\bb E}_{R\ovl\otimes N}. 
\end{equation}
Now define $\psi = \phi - \partial\tilde\alpha - \partial\tilde\beta$, a 2-cocycle equivalent to $\phi$. We verify the desired properties for $\psi$. Separate normality is clear from the choices of $\alpha$ and $\beta$, so $\psi$ satisfies (C1).

Since $\phi$ satisfies (C2), we have $\phi(I,I) = 0$. Then the cocycle identity
\begin{equation}\label{eq3.7}
I\alpha(I) - \alpha(I) + \alpha(I)I = \phi(I,I) = 0
\end{equation}
gives $\alpha(I)=0$, and modularity then implies that $\tilde\alpha|_{R\ovl\otimes S} = \alpha|_{R\ovl\otimes S}=0$, with a similar result for $\tilde\beta$. A straightforward calculation then shows that $\partial\tilde \alpha(x,y)$ and $\partial\tilde\beta(x,y)$ are both 0 whenever at least one of $x$ and $y$ lies in $R\ovl\otimes S$. Thus $\psi$ satisfies (C2) and hence (C3). It remains to show that \eqref{eq3.3} is satisfied. We consider only the relation $\psi(m_1\otimes I, m_2\otimes I)=0$ for $m_1,m_2\in M$, since the argument for the second is identical.

For $m\in M$,
\begin{equation}\label{eq3.8}
 \tilde\beta(m\otimes I) = \beta({\bb E}_R(m)\otimes I) = 0,
\end{equation}
since $\beta$ vanishes on $R\ovl\otimes S$,  and thus $\partial\tilde\beta|_{M\otimes I}=0$.  Consequently $\psi|_{M\otimes I} = \phi|_{M\otimes I} - \partial\tilde\alpha|_{M\otimes I}$, and we determine the latter term. For $m\in M$,
\begin{equation}\label{eq3.9}
 \tilde\alpha(m\otimes I) = \alpha(m\otimes {\bb E}_S(I)) = \alpha(m\otimes I),
\end{equation}
so
\begin{equation}\label{eq3.10}
 \partial\tilde\alpha(m_1\otimes I, m_2\otimes I) = \partial\alpha(m_1\otimes I, m_2\otimes I) = \phi(m_1 \otimes I, m_2\otimes I), \quad m_1,m_2\in M,
\end{equation}
since $\phi=\partial\alpha$ on $M\ovl\otimes S$. This shows that \eqref{eq3.3} holds.
\end{proof}

In light of this lemma, we may henceforth assume that the 2-cocycle $\phi$ on $M\ovl\otimes N$ not only satisfies (C1)--(C3) but also condition \eqref{eq3.3}. We will need to make use of the basic construction for an inclusion $P\subseteq Q$ of finite von Neumann algebras, where $Q$ has a specified normal faithful trace $\tau$. Then $Q$ acts on the Hilbert space $L^2(Q,\tau)$, which we abbreviate to $L^2(Q)$, and its commutant is $JQJ$,  where $J$ is the canonical conjugation. We will use $J$ for all such conjugations, which should be clear from the context. The Hilbert space projection of $L^2(Q)$ onto $L^2(P)$ is denoted by $e_p$, and the basic construction $\langle Q,e_p\rangle$ is the von Neumann algebra generated by $Q$ and $e_p$. Since $\langle Q,e_p\rangle' = JPJ$, \cite{J}, it is clear that $\langle Q,e_p\rangle$ is hyperfinite precisely when $P$ has this property.

For the inclusions $A\ovl\otimes B \subseteq R \ovl\otimes S \subseteq M\ovl\otimes N$, we obtain an inclusion $\langle M\ovl\otimes N, e_{R\ovl\otimes S}\rangle \subseteq \langle M\ovl\otimes N, e_{A\ovl\otimes B}\rangle$ of hyperfinite von Neumann algebras. Since $J(A\ovl\otimes B)J$ is a masa in $J(M\ovl\otimes N)J$, the general theory of extended cobounding, \cite{KR2} or \cite{SSbook}, allows us to find a bounded $(R\ovl\otimes S)$-modular map $\xi\colon \ M\ovl\otimes N\to \langle M\ovl\otimes N, e_{A\ovl\otimes B}\rangle$ so that $\phi = \partial\xi$. Hyperfiniteness gives a conditional expectation ${\bb E}\colon \ \langle M\ovl\otimes N, e_{A\ovl\otimes B}\rangle \to \langle M\ovl\otimes N, e_{R\ovl\otimes S}\rangle$, and the $(R\ovl\otimes S)$-modular composition $\gamma = {\bb E}\circ\xi\colon \ M\ovl\otimes N\to \langle M\ovl\otimes N, e_{R\ovl\otimes S}\rangle$ also has the property that $\phi = \partial\gamma$. Moreover, the results of \cite{JKR} allow us to further assume that $\gamma$ is normal.

We now introduce three auxiliary linear maps. At the outset these are not obviously bounded, and so can only be defined on the algebraic tensor product $M\otimes N$. We define $f,g\colon \ M\otimes N\to \langle M\ovl\otimes N, e_{R\ovl\otimes S}\rangle$ and $h\colon \ M\otimes N\to M\ovl\otimes N$  on elementary tensors $m\otimes n \in M\otimes N$ by
\begin{align}\label{eq3.11}
 f(m\otimes n) &= \phi(m\otimes I, I\otimes n) + \gamma(m\otimes n),\\
\label{eq3.12}
g(m\otimes n) &= \phi(I\otimes n, m\otimes I) + \gamma(m\otimes n),\\
\label{eq3.13}
h(m\otimes n) &= g(m\otimes n)-f(m\otimes n) = \phi(I\otimes n, m\otimes I) - \phi(m\otimes I, I\otimes n).
\end{align}
The next lemma lists some basic properties of these maps.

\begin{lem}\label{lem3.2}
The following properties hold:
\begin{itemize}
 \item[\rm (i)] The restrictions $\gamma|_{M\ovl\otimes S}$ and $\gamma|_{R\ovl\otimes N}$ are completely bounded derivations, spatially implemented by elements of $\langle M\ovl\otimes N, e_{R\ovl\otimes S}\rangle$.
\item[\rm (ii)] The restrictions $f|_{M\otimes I}$, $f|_{I\otimes N}$, $g|_{M\otimes I}$, and $g|_{I\otimes N}$ are equal to the respective restrictions of $\gamma$ to these subalgebras, and are all bounded derivations spatially implemented by elements of $\langle M\ovl\otimes N, e_{R\ovl\otimes S}\rangle$.
\item[\rm (iii)] The restrictions $h|_{M\otimes I}$ and $h|_{I\otimes N}$ are both 0.
\end{itemize}
\end{lem}

\begin{proof}
 (i)\qquad We consider only $\gamma|_{M\ovl\otimes S}$, the other case being similar. Since $\phi|_{M\otimes I} = 0$, from \eqref{eq3.3}, and $\phi = \partial\gamma$, we see that $\gamma|_{M\otimes I}$ is a derivation. The $(R\ovl\otimes S)$-modularity then implies that $\gamma|_{M\otimes S}$ is a derivation, with the same conclusion for $\gamma|_{M\ovl\otimes S}$ by normality of $\gamma$. Since $\gamma$ is, in particular, $(I\otimes S)$-modular, complete boundedness of $\gamma|_{M\ovl\otimes S}$ follows from Lemma \ref{lem2.1}. Thus $\gamma|_{M\ovl\otimes S}$ is implemented by an operator $t\in B(L^2(M\ovl\otimes N))$, from Lemma \ref{lem2.3}. By hyperfiniteness of $\langle M\ovl\otimes N, e_{R\ovl\otimes S}\rangle$, there is a conditional expectation ${\bb E}$ of $B(L^2(M\ovl\otimes N))$ onto this subalgebra, and so $\gamma|_{M\ovl\otimes S}$ is also implemented by ${\bb E}(t)\in \langle M\ovl\otimes N, e_{R\ovl\otimes S}\rangle$.

\n (ii)\qquad From \eqref{eq3.11} and (C2),
\begin{align}
 f(m\otimes I) &= \phi(m\otimes I, I\otimes I) + \gamma(m\otimes I)\notag\\
\label{eq3.14}
&= \gamma(m\otimes I),\qquad m\in M,
\end{align}
so $f|_{M\otimes I} = \gamma|_{M\otimes I}$ is a derivation on $M\otimes I$ spatially implemented by an element of $\langle M\ovl\otimes N, e_{R\ovl\otimes S}\rangle$ from (i). The other three restrictions are handled similarly.

\n (iii)\qquad From (ii)
\begin{equation}\label{eq3.15}
 h|_{M\otimes I} = g|_{M\otimes I} - f|_{M\otimes I} = \gamma|_{M\otimes I} - \gamma|_{M\otimes I}  = 0,
\end{equation}
with a similar result for $h|_{I\otimes N}$.
\end{proof}

\begin{pro}\label{pro3.3}
The map $f$ of \eqref{eq3.11} is a derivation on $M\otimes N$.
\end{pro}

\begin{proof}
For $m\in M$ and $n\in N$, Lemma \ref{lem3.2} (ii) gives
\begin{align}
(m\otimes I) f(I\otimes n) + f(m\otimes I)(I\otimes n) &= (m\otimes I) \gamma(I\otimes n) + \gamma(m\otimes I)(I\otimes n)\notag\\
&= [(m\otimes I)\gamma(I\otimes n) - \gamma(m\otimes n) + \gamma(m\otimes I)(I\otimes n)]\notag\\
&\quad  + \gamma(m\otimes n)\notag\\
&= \phi(m\otimes I, I\otimes n) + \gamma(m\otimes n)\notag\\
\label{eq3.16}
&= f(m\otimes n),
\end{align}
using $\phi=\partial\gamma$. A similar calculation leads to
\begin{equation}\label{eq3.17}
 (I\otimes n) g(m\otimes I) + g(I\otimes n)(m\otimes I) = g(m\otimes n).
\end{equation}
We now use \eqref{eq3.16} and \eqref{eq3.17} to calculate $\partial f$ on pairs of elementary tensors, noting that $f$ is a derivation on $M\otimes I$ and $I\otimes N$. For $m_1,m_2\in M$, $n_1,n_2\in N$,
\begin{align}
\partial f(m_1\otimes n_1, m_2\otimes n_2) &= (m_1\otimes n_1) f(m_2\otimes n_2) -f(m_1m_2\otimes n_1n_2) + f(m_1\otimes n_1)(n_2\otimes n_2)\notag\\
&= (m_1\otimes n_1) [(m_2\otimes I) f(I\otimes n_2) + f(m_2\otimes I)(I\otimes n_2)]\notag\\
&\quad - [(m_1m_2\otimes I) f(I\otimes n_1n_2) + f(m_1m_2\otimes I)(I\otimes n_1n_2)]\notag\\
&\quad + [(m_1\otimes I) f(I\otimes n_1) + f(m_1\otimes I)(I\otimes n_1)](m_2\otimes n_2)\notag\\
&= (m_1m_2\otimes n_1) f(I\otimes n_2) + (m_1\otimes n_1) f(m_2\otimes I)(I\otimes n_2)\notag\\
&\quad - (m_1m_2\otimes I) [(I\otimes n_1) f(I\otimes n_2) + f(I\otimes n_1)(I\otimes n_2)]\notag\\
&\quad - [(m_1\otimes I)f(m_2\otimes I) + f(m_1\otimes I)(m_2\otimes I)] (I\otimes n_1n_2)\notag\\
&\quad + (m_1\otimes I) f(I\otimes n_1)(m_2\otimes n_2) + f(m_1\otimes I) (m_2\otimes n_1n_2)\notag\\
&= (m_1m_2\otimes n_1) f(I\otimes n_2) + (m_1\otimes n_1) f(m_2\otimes I)(I\otimes n_2)\notag\\
&\quad - (m_1m_2\otimes n_1) f(I\otimes n_2) - (m_1m_2\otimes I) f(I\otimes n_1)(I\otimes n_2)\notag\\
&\quad - (m_1\otimes I) f(m_2\otimes I)(I\otimes n_1n_2) - f(m_1\otimes I)(m_2\otimes n_1n_2)\notag\\
&\quad + (m_1\otimes I) f(I\otimes n_1)(m_2\otimes n_2) + f(m_1\otimes I)(m_2\otimes n_1n_2)\notag\\
&= (m_1\otimes I) [(I\otimes n_1) f(m_2\otimes I) + f(I\otimes n_1)(m_2\otimes I)] (I\otimes n_2)\notag\\
\label{eq3.18} 
&\quad - (m_1\otimes I) [(m_2\otimes I)f(I\otimes n_1) + f(m_2\otimes I)(I\otimes n_1)] (I\otimes n_2).
\end{align}
Recalling that $f,g$ and $\gamma$ agree on $M\otimes I$ and $I\otimes N$, while $\phi  = \partial\gamma$, \eqref{eq3.18} becomes
\begin{align}
\partial f(m_1\otimes n_1, m_2\otimes n_2) &= (m_1\otimes I) [(I\otimes n_1) \gamma(m_2\otimes I) + \gamma(I\otimes n_1)(m_2\otimes I)] (I\otimes n_2)\notag\\
&\quad - (m_1\otimes I) [(m_2\otimes I) \gamma(I\otimes n_1) + \gamma(m_2\otimes I) (I\otimes n_1)] (I\otimes n_2)\notag\\
&= (m_1\otimes I) [\phi(I\otimes n_1, m_2\otimes I) + \gamma(m_2\otimes n_1)] (I\otimes n_2)\notag\\
&\quad - (m_1\otimes I) [\phi(m_2\otimes I, I\otimes n_1) + \gamma(m_2\otimes n_1)] (I\otimes n_2)\notag\\
&= (m_1\otimes I) [g(m_2\otimes n_1) - f(m_2\otimes n_1)] (I\otimes n_2)\notag\\
\label{eq3.19}
&= (m_1\otimes I) h(m_2\otimes n_1) (I\otimes n_2).
\end{align}
Here we have used the relations \eqref{eq3.11}--\eqref{eq3.13}. Now define $F=\partial f$. Then \eqref{eq3.19} is 
\begin{equation}\label{eq3.20}
 F(m_1\otimes n_1, m_2\otimes n_2) = (m_1\otimes I) h(m_2\otimes n_1)(I\otimes n_2).
\end{equation}
The identity $\partial F=0$ for the triple $(I\otimes n_1, m_2\otimes I, m_3\otimes I)$ yields
\begin{equation}\label{eq3.21}
 (m_2\otimes n_1) h(m_3\otimes I) - (m_2\otimes I) h(m_3\otimes n_1) + h(m_2m_3\otimes n_1) - h(m_2\otimes n_1) (m_3\otimes I)=0,
\end{equation}
and so
\begin{equation}\label{eq3.22}
 h(m_2m_3\otimes n_1) = (m_2\otimes I) h(m_3\otimes n_1) + h(m_2\otimes n_1)(m_3\otimes I)
\end{equation}
since $h|_{M\otimes I}=0$. It follows from \eqref{eq3.22} that, for each fixed $n_1\in N$, the map $\delta(m\otimes I) = h(m\otimes n_1)$, $m\in M$, defines a derivation of $M\otimes I$ into $M\ovl\otimes N$. Since both algebras are finite, $\delta$ is implemented by an element $a\in M\ovl\otimes N$ by Lemma \ref{lem2.4}. For $r\in R$,
\begin{equation}\label{eq3.23}
 \delta(r\otimes I) = h(r\otimes n_1) = \phi(I\otimes n_1, r\otimes I) - \phi(r\otimes I, I\otimes n_1) = 0,
\end{equation}
from \eqref{eq3.13} and (C2). Thus $a\in (R\otimes I)'\cap (M\ovl\otimes N) = Z(M)\ovl\otimes N$, so $a$ commutes with $M\otimes I$. We conclude that $h(m\otimes n_1) = 0$ for $m\in M$. Since $n_1\in N$ was arbitrary, $h=0$ and, from \eqref{eq3.19}, $\partial f= 0$. This shows that $f$ is a derivation on the algebraic tensor product $M\otimes N$.
\end{proof}

\begin{pro}\label{pro3.4}
There exists a bounded normal map $\xi\colon \ M\ovl\otimes N\to M\ovl\otimes N$ such that
\begin{equation}\label{eq3.24}
 \xi(m\otimes n) = \phi(m\otimes I, I\otimes n), \qquad m\in M,\  n\in N.
\end{equation}
\end{pro}

\begin{proof}
From Proposition \ref{pro3.3}, $f$ is a derivation on $M\otimes N$ with values in $\langle M\ovl\otimes N, e_{R\ovl\otimes S}\rangle = \langle M,e_R\rangle \ovl\otimes \langle N,e_S\rangle$. By Lemma \ref{lem3.2} (ii), $f|_{M\otimes I}$ is a completely bounded derivation implemented by an element $t\in\langle M,e_R\rangle \ovl\otimes \langle N,e_S\rangle$. Define a derivation $\delta\colon \ M\otimes N\to \langle M,e_R\rangle \ovl\otimes \langle N,e_S\rangle$ by
\begin{equation}\label{eq3.25}
\delta(m\otimes n) = f(m\otimes n) - [t(m\otimes n) - (m\otimes n)t],\qquad m\in M,\  n\in N.
\end{equation}
Then $\delta|_{M\otimes I} = 0$ from \eqref{eq3.25}, so $\delta$ is $(M\otimes I)$-modular. From Lemma \ref{lem3.2} (ii), $f|_{1\otimes N}$ is a derivation implemented by an element of $\langle M,e_R\rangle \ovl\otimes \langle N,e_S\rangle$, so from \eqref{eq3.25} there is an element $b$ in this algebra such that
\begin{equation}\label{eq3.26}
\delta(I\otimes n)  = b(I\otimes n) - (I\otimes n)b,\qquad n\in N.
\end{equation}
The $(M\otimes I)$-modularity of $\delta$ shows that
\begin{equation}\label{eq3.27}
 (m\otimes I) \delta(I\otimes n) = \delta(m\otimes n) = \delta(I\otimes n)(m\otimes  I),\qquad m\in M,\  n\in N,
\end{equation}
and we conclude that the range of $\delta|_{I\otimes N}$ lies in $(M\otimes I)'\cap \langle M,e_R\rangle \ovl\otimes \langle N,e_S\rangle$. This algebra is $(M'\cap \langle M,e_R\rangle) \ovl\otimes \langle N,e_S\rangle$, equal to $(JMJ\cap (JRJ)') \ovl\otimes \langle N,e_S\rangle$, and in turn equal to $(JZ(M)J)\ovl\otimes \langle N,e_S\rangle$. The latter algebra is hyperfinite, so if we take a conditional expectation onto it and apply this to \eqref{eq3.26}, then we conclude that the element $b$ of \eqref{eq3.26} may be assumed to lie in $(JZ(M)J)\ovl\otimes \langle N,e_S\rangle$. Then $b$ commutes with $M\otimes I$, so 
\begin{align}
 \delta(m\otimes n) &= (m\otimes I) \delta(I\otimes n) = (m\otimes I) [b(I\otimes n) - (I\otimes n)b]\notag\\
\label{eq3.28}
&= b(m\otimes n) - (m\otimes n)b,\qquad m\in M,\  n\in N.
\end{align}
Thus $\delta$ has a unique bounded normal extension to $M\ovl\otimes N$, and \eqref{eq3.25} shows that the same is then true for $f$. Since $\xi=f-\gamma$ on $M\otimes N$ from \eqref{eq3.11}, and $\gamma$ is already bounded and normal on $M\ovl\otimes N$, this gives a bounded normal extension of $\xi$ to $M\ovl\otimes N$.
\end{proof}

\begin{rem}\label{rem3.5}
Equation \eqref{eq3.25} shows that the derivation $f$ on $M\otimes N$ has a unique normal
extension to $M\ovl\otimes N$. Taking ultraweak limits in the equation
\begin{equation}
f(xy)=xf(y)+f(x)y,\qquad x,y\in M\otimes N,
\end{equation}
shows that this extension is also a derivation on 
$M\ovl\otimes N$.$\hfill\square$
\end{rem}
We now come to the main result of this section.

\begin{thm}\label{thm3.5}
Let $M$ and $N$ be separable type {\rm II}$_1$ von Neumann algebras. Then
\begin{equation}\label{eq3.29}
 H^2(M\ovl\otimes N, M\ovl\otimes N) = 0.
\end{equation}
\end{thm}

\begin{proof}
We have already reduced consideration of a general cocycle $\phi$ to one which satisfies (C1)--(C3) and \eqref{eq3.3}. With the previously established notation, Proposition \ref{pro3.4} and Remark \ref{rem3.5} show that $\xi$ and $f$ have bounded normal extensions from $M\otimes N$ to $M\ovl\otimes N$.  Using the same letters for the extensions, we see that $\xi$ maps $M\ovl\otimes N$ to itself, while  $f$ is a derivation on $M\ovl\otimes N$ from Remark \ref{rem3.5}. Thus
\begin{equation}\label{eq3.30}
 \phi = \partial \gamma = \partial f - \partial\xi = \partial(-\xi)
\end{equation}
on $(M\ovl\otimes N) \times (M\ovl\otimes N)$. This shows that $\phi$ is a coboundary with respect to the bounded linear map $-\xi$, proving the result.
\end{proof}

\begin{rem}\label{rem3.6}
We will require one more piece of information about maps $\xi$ on $M \ovl\otimes N$ for which $\phi =\partial\xi$, namely that they can be chosen so that $\|\xi\| \le C\|\phi\|$ for an absolute constant $C$. The argument is already essentially in \cite[Lemma 6.5.1]{SSbook}, so we only sketch it here.

If no such $C$ existed, then it would be possible to find separable type II$_1$ algebras $M_n$ and $N_n$ for $n\ge 1$, and cocycles $\phi_n$ on $M_n\ovl\otimes N_n$ of norm 1 so that any $\xi_n$ satisfying $\phi_n=\partial \xi_n$ necessarily had norm at least $n$. Form separable algebras $M = \bigoplus\limits^\infty_{n=1} M_n$ and $N = \bigoplus\limits^\infty_{n=1} N_n$ and define a cocycle $\phi$ on $M\ovl\otimes N$ by
\[
 \phi(m_i\otimes n_j, \widetilde m_k \otimes \widetilde n_\ell)  = \phi_i(m_i\otimes n_i, \widetilde m_i\otimes \widetilde n_i)
\]
when $i=j=k=\ell$, and 0 otherwise. By Theorem \ref{thm3.5} there exists a bounded map $\xi$ on $M\ovl\otimes N$ so that $\phi=\partial\xi$ (which can be assumed to be $Z(M\ovl\otimes N)$-modular), but this would then contradict the lower bounds on $\|\xi_n\|$ by restricting $\xi$ to the component algebras.$\hfill\square$
\end{rem}

\section{The general case}\label{sec4}

\indent 

The techniques of Section  \ref{sec3} relied heavily on the existence of hyperfinite subalgebras whose relative commutants are the center, and these are only guaranteed to exist in the separable case. We will use Theorem \ref{thm3.5} and Remark \ref{rem3.6} to derive the general result, but we require some preliminary lemmas. A complication for a general type II$_1$ von Neumann algebra $M$ is that it need not have a faithful normal trace. However, a standard maximality argument gives a family of central projections $p_\lambda$ with sum $I$ such that each $Mp_\lambda$ has such a trace. Until we reach Theorem \ref{thm4.5}, we restrict attention to those algebras which do have faithful normal traces.

\begin{lem}\label{lem4.1}
Let $M$ and $N$ be type {\rm II}$_1$ von Neumann algebras with faithful normal unital traces $\tau_M$ and $\tau_N$ respectively, and let $Q\subseteq M\ovl\otimes N$ be a separable von Neumann subalgebra. Then there exist separable type {\rm II}$_1$ von Neumann subalgebras $M_0\subseteq M$ and $N_0\subseteq N$ such that $Q\subseteq M_0 \ovl\otimes N_0$.
\end{lem}

\begin{proof}
 We may certainly assume that $Q$ contains arbitrarily large matrix subalgebras of $M$ and  $N$, and this will guarantee that the $M_0$ and $N_0$ that we construct are type II$_1$.

Let $\tau=\tau_M\otimes \tau_N$ be a faithful normal unital trace on $M\ovl\otimes N$, and fix a countable ultraweakly dense sequence $\{q_n\}^\infty_{n=1}$ in the unit ball of $Q$. The identity map of $M\ovl\otimes N$ into $L^2(M\ovl\otimes N)$ is ultraweakly-to-weakly continuous so the unit ball of $M\otimes N$, which is ultraweakly dense in that of $M\ovl\otimes N$ by the Kaplansky density theorem, also has this property in the $\|\cdot\|_2$-norm. Thus each $q_n$ is the $\|\cdot\|_2$-limit of sequences from $M\otimes N$, each element of which is a finite sum  of elementary tensors. Let $M_0$ be (respectively $N_0$) the von Neumann algebra generated by the first (respectively second) entries in all of these elementary tensors. Each is separable. Then $L^2(Q) \subseteq L^2(M_0\ovl\otimes N_0)$ and so $Q \subseteq M_0 \ovl\otimes N_0$ by considering the conditional expectation of $M\ovl\otimes N$ onto $M_0\ovl\otimes N_0$, which also defines the Hilbert space projection of $L^2(M\ovl\otimes N)$ onto $L^2(M_0\ovl\otimes N_0)$ (see the proof of \cite[Lemma 2.2]{Cam}).
\end{proof}

\begin{lem}\label{lem4.2}
Let $M$ and $N$ be type {\rm II}$_1$ von Neumann algebras with faithful normal unital traces $\tau_M$ and $\tau_N$ respectively. Let $\phi$ be a separately normal bounded bilinear map from $(M\ovl\otimes N) \times (M\ovl\otimes N)$ to $ M\ovl\otimes N$. Given a finite set $F\subseteq M\ovl\otimes N$, there exist separable type {\rm II}$_1$ von Neumann subalgebras $M_F\subseteq M$ and $N_F\subseteq N$ such that $F\subseteq M_F\ovl\otimes N_F$ and $\phi$ maps $(M_F\ovl\otimes N_F)\times (M_F\ovl\otimes N_F)$ to $(M_F\ovl\otimes N_F)$.
\end{lem}

\begin{proof}
We apply Lemma \ref{lem4.1} repeatedly. Let $Q_0$ be the von Neumann generated by $F$ and choose separable von Neumann algebras $M_0\subseteq M$ and $N_0\subseteq N$ so that $Q_0\subseteq M_0\ovl\otimes N_0$. Then let $Q_1$ be the separable von Neumann algebra generated by $M_0\ovl\otimes N_0$ and the range of $\phi|_{M_0\ovl\otimes N_0}$. Now choose separable von Neumann algebras so that $Q_1\subseteq M_1\ovl\otimes N_1$. By construction, $\phi$ maps $(M_0\ovl\otimes N_0)\times (M_0\ovl\otimes N_0)$ into $M_1\ovl\otimes N_1$. Continuing in this way, we obtain an ascending sequence $\{M_i\ovl\otimes N_i\}^\infty_{i=0}$ of separable von Neumann algebras so that $\phi$ maps $(M_i\ovl\otimes N_i)\times (M_i\ovl\otimes N_i)$ into $M_{i+1}\ovl\otimes N_{i+1}$. Define $M_F$ and $N_F$ as the respective ultraweak closures of $\bigcup\limits^\infty_{i=0} M_i$ and $\bigcup\limits^\infty_{i=0} N_i$. Then separate normality shows that $\phi$ maps $(M_F\ovl\otimes N_F)\times (M_F\ovl\otimes N_F)$ into $M_F\ovl\otimes N_F$ as required.
\end{proof}

The next result is a special case of the subsequent main result.

\begin{pro}\label{pro4.3}
Let $M$ and $N$ be type {\rm II}$_1$ von Neumann algebras with faithful normal unital traces. Then
\begin{equation}\label{eq4.1}
 H^2(M\ovl\otimes N, M\ovl\otimes N) = 0.
\end{equation}
\end{pro}

\begin{proof}
Theorem 3.3.1 of \cite{SSbook} allows us to restrict attention to a separately normal 2-cocycle $\phi$ on $M\ovl\otimes N$. For each finite subset $F$ of $M\ovl\otimes N$, let $M_F$ and $N_F$ be the separable von Neumann subalgebras constructed in Lemma \ref{lem4.2}, so that $\phi$ maps $(M_F\ovl\otimes N_F) \times (M_F\ovl\otimes N_F)$ to $M_F\ovl\otimes N_F$. Let $\phi_F$ be the restriction of $\phi$ to this subalgebra. By Theorem \ref{thm3.5}, there is a bounded linear map $\xi_F\colon \ M_F\ovl\otimes N_F\to M_F\ovl\otimes N_F$ so that $\phi_F = \partial\xi_F$, and Remark \ref{rem3.6} allows us to assume a uniform bound on $\|\xi_F\|$ independent of $F$. The construction of a bounded map $\xi\colon \ M\ovl\otimes N\to M\ovl\otimes N$ such that $\phi = \partial\xi$ now follows the proof of \cite[Theorem 6.5.3]{SSbook}.
\end{proof}

\begin{rem}\label{rem4.4}
An examination of the proof of \cite[Theorem 6.5.3]{SSbook} combined with Remark \ref{rem3.6} shows the existence of an absolute constant $K$ so that, under the hypotheses of Proposition \ref{pro4.3}, to each 2-cocycle $\phi$ on $M\ovl\otimes N$ there corresponds a bounded map $\xi$ on $M\ovl\otimes N$ satisfying $\phi=\partial\xi$ and $\|\xi\| \le K\|\phi\|$.$\hfill\square$
\end{rem}

The final step is to remove the hypothesis of faithful traces from Proposition \ref{pro4.3}.

\begin{thm}\label{thm4.5}
Let $M$ and $N$ be type {\rm II}$_1$ von Neumann algebras. Then
\begin{equation}\label{eq4.2}
 H^2(M\ovl\otimes N, M\ovl\otimes N)=0.
\end{equation}
\end{thm}

\begin{proof}
As noted earlier, there are orthogonal sets of central projections $p_\lambda\in Z(M)$ and $q_\mu\in Z(N)$, each summing to $I$, so that $Mp_\lambda$ and $Nq_\mu$ have faithful normal unital traces. Given a separately normal 2-cocycle $\phi$ on $M\ovl\otimes N$, \cite[Theorem 3.2.7]{SSbook} allows us to assume that it is $Z(M\ovl\otimes N)$-multimodular. Thus the restriction $\phi_{\lambda,\mu}$ of $\phi$ to $Mp_\lambda \ovl\otimes Nq_\mu$ maps back to this algebra. By Proposition \ref{pro4.3} and Remark \ref{rem4.4}, there are maps $\xi_{\lambda,\mu}\colon \ Mp_\lambda \ovl\otimes Nq_\mu\to Mp_\lambda \ovl\otimes Nq_\mu$ so that $\phi_{\lambda,\mu} = \partial\xi_{\lambda,\mu}$ with a uniform bound on $\|\xi_{\lambda,\mu}\|$. This allows us to define a bounded map $\xi\colon \ M\ovl\otimes N\to M\ovl\otimes N$ by $\xi = \bigoplus\limits_{\lambda,\mu} \xi_{\lambda,\mu}$, and then $\phi=\partial\xi$.
\end{proof}

\end{document}